\documentclass{amsart}

\usepackage{sty/preamble}

\newcommand{\ack}{\subsection*{Acknowledgment}}

\title{Multicategories Model All Connective Spectra}

\authorinfoNJDY

\hypersetup{pdfauthor=\authors}

\date{18 March 2022}


\begin{document}

\begin{abstract}
  There is a free construction from multicategories to permutative categories, left adjoint to the endomorphism multicategory construction.
The main result shows that these functors induce an equivalence of homotopy theories.
This result extends a similar result of Thomason, that permutative categories model all connective spectra.

\end{abstract}

\maketitle

\tableofcontents
\section{Introduction}\label{sec:intro}

Thomason proved in \cite{thomason} that symmetric monoidal categories model all connective spectra, in the sense that there is an equivalence of homotopy categories between that of $\SMC$ and of $\Sp_{\ge 0}$ with respect to stable equivalences.  See \cite{mandell_inverseK} and \cite{gjo1} for further refinements of this result to an equivalence of homotopy \emph{theories}
\begin{equation}\label{eq:perm-spge0-hty-thy}
\big( \PermCat, \cS \big) \sim
\big( \Sp_{\ge 0}, \cS \big)
\end{equation}
in the sense of \cref{definition:hty-thy} below.  In \eqref{eq:perm-spge0-hty-thy},
\begin{itemize}
\item $\PermCat$ denotes the category of small permutative categories and strict monoidal functors,
\item $\Sp_{\ge 0}$ denotes the category of connective symmetric spectra, and
\item $\cS$ denotes the class of stable equivalences.
\end{itemize}
The equivalence is induced by Segal's $K$-theory construction \cite{segal}
\[
K\cn \PermCat \to \Sp_{\ge 0}.
\]
We give further explanations of the relevant background material in \cref{sec:permcat,sec:multicat,sec:multicats-model}. 

In this article we show that there is an equivalence of homotopy theories
\begin{equation}\label{eq:multicat-perm-hty-thy}
\big( \Multicat, \cS \big) \sim
\big( \PermCat, \cS \big).
\end{equation}
The equivalence is induced by a free functor
\[
F \cn \Multicat \to \PermCat
\]
that we describe in \cref{sec:free-perm,sec:free-perm-adj} below.
We prove the homotopy equivalence \eqref{eq:multicat-perm-hty-thy} in \cref{theorem:F-End-hthy-equiv}.
Combining \eqref{eq:multicat-perm-hty-thy} with \eqref{eq:perm-spge0-hty-thy} above gives the following main result.
\begin{theorem}\label{theorem:multicat-perm-hty-thy}
  There is an equivalence of homotopy theories
  \[
  \big( \Multicat, \cS \big) \sim
  \big( \Sp_{\ge 0}, \cS \big)
  \]
  induced by
  \[
  \Multicat \fto{F} \PermCat \fto{K} \Sp_{\ge 0}.
  \]
\end{theorem}

\subsection*{Motivation}
Multicategories, as a context for stable homotopy, have several advantages related to symmetric monoidal closed structure known as the Boardman-Vogt tensor product \cite{boardman-vogt}.
This monoidal product does not restrict to $\PermCat$, but it does provide $\PermCat$ with the structure of a multicategory.
This perspective appears in work of Elmendorf-Mandell \cite{elmendorf-mandell,elmendorf-mandell-perm}, which goes on to develop a multifunctorial $K$-theory extending that of Segal.
A detailed explanation of the relevant theory is given in \cite[Part~2]{cerberusIII}.  In particular, see \cite[10.3.32 and 10.8]{cerberusIII}.

The category of small Waldhausen categories, useful in the study of algebraic $K$-theory, also admits the structure of a closed multicategory that is not symmetric monoidal.  See Zakharevich \cite{zakharevich} for a proof and further explanation.  Work of Bohmann-Osorno \cite{bohmann_osorno} develops a multifunctorial $K$-theory for Waldhausen categories similar to that of Elmendorf-Mandell.

The conclusion of \cref{theorem:multicat-perm-hty-thy} shows that, although richer in algebraic structure, multicategories model the same homotopy theory as that of permutative categories, namely, connective stable homotopy theory.  Thus, while the category of small multicategories provides a context for multiplicative homotopy theory, it does not introduce exotic or spurious homotopy types.

\subsection*{Outline}
\Cref{sec:hty-thy,sec:permcat,sec:multicat} fix notation and terminology by giving relevant background for complete Segal spaces, permutative categories, and multicategories.  In \cref{sec:free-perm} we define the free construction $F$, together with unit and counit natural transformations.  In 
\cref{sec:free-perm-adj} we show that these define an adjunction of 2-categories.  \Cref{sec:multicats-model} contains the definitions of stable equivalences and, in \cref{theorem:F-End-hthy-equiv}, the proof that $F$ is an equivalence of homotopy theories.

\ack
We thank the referee for their careful reading of this paper, leading to a number of improvements.

\section{Equivalences of Homotopy Theories}\label{sec:hty-thy}

In this section we review the theory of complete Segal spaces due to Rezk \cite{rezk-homotopy-theory}.
An equivalence of homotopy theories (\cref{definition:hty-thy}) is an equivalence of fibrant replacements in the complete Segal space model structure.
For further context and development we refer the reader to \cite{dwyer-kan,hirschhorn,toen-axiomatisation,barwick-kan}.

\subsection*{Complete Segal Spaces}

\begin{definition}\label{definition:css}
  A bisimplicial set $X$ is a \emph{complete Segal Space} if
  \begin{itemize}
  \item it is fibrant in the Reedy model structure on bisimplicial sets,
  \item for each $n \ge 2$ the Segal map
    \[
    X(n) \to X(1) \times_{X(0)} \cdots \times_{X(0)} X(1)
    \]
    is a weak equivalence of simplicial sets, and
  \item the morphism
    \begin{equation}\label{eq:css-char}
    X(0) \iso \Map(\De[0],X) \to \Map(E,X)
    \end{equation}
    is a weak equivalence of simplicial sets, where $E$ is the discrete nerve of the category consisting of two isomorphic objects and \eqref{eq:css-char} is induced by the unique morphism $E \to \De[0]$. \dqed
  \end{itemize}
\end{definition}

\begin{remark}
  The definition of complete Segal space given above is equivalent to that given in \cite[Section~6]{rezk-homotopy-theory} by \cite[6.4]{rezk-homotopy-theory}.
\end{remark}

\begin{theorem}[{\cite[7.2]{rezk-homotopy-theory}}]\label{theorem:css-fibrant}
  There is a simplicial closed model structure on the category of bisimplicial sets, called the \emph{complete Segal space model structure}, that is given as a left Bousfield localization of the Reedy model structure and in which the fibrant objects are precisely the complete Segal spaces.
\end{theorem}

\subsection*{Relative Categories}

\begin{definition}\label{definition:rel-cat}
  A \emph{relative category} is a pair $(\C,\cW)$ consisting of a category $\C$ and a subcategory $\cW$ containing all of the objects of $\C$.
  A \emph{relative functor}
  \[
  F\cn (\C,\cW) \to (\C',\cW')
  \]
  is a functor from $\C$ to  $\C'$ that sends morphisms of $\cW$ to those of $\cW'$.
\end{definition}

\begin{definition}\label{definition:rel-cat-pow}
  Suppose $(\C,\cW)$ is a relative category and $\A$ is another category.
  We let
  \[
  (\C,\cW)^\A
  \]
  denote the subcategory of $\C^\A$ whose objects are functors $\A \to \C$ and whose morphisms are those natural transformations with components in $\cW$.
\end{definition}

\begin{definition}\label{definition:classification-diagram}
  Suppose $(\C,\cW)$ is a relative category.
  The \emph{classification diagram} of $(\C,\cW)$ is the bisimplicial set
  \[
  \Ncl(\C,\cW) = \Ner\big((\C,\cW)^{\De[\bdot]}\big)
  \]
  given by
  \[
  n \mapsto \Ner\big((\C,\cW)^{\De[n]}\big)
  \]
  where $\De[n]$ denotes the category consisting of $n$ composable arrows.
\end{definition}

\begin{definition}\label{definition:hty-thy}
  Suppose $(\C,\cW)$ is a relative category.  We say that a bisimplicial set $R\Ncl(\C,\cW)$ is a \emph{homotopy theory of $(\C,\cW)$} if it is a fibrant replacement of $\Ncl(\C,\cW)$ in the complete Segal space model structure.
  We say that a relative functor
  \[
  F \cn (\C,\cW) \to (\C',\cW')
  \]
  is an \emph{equivalence of homotopy theories} if the induced morphism $R\Ncl F$ between homotopy theories is a weak equivalence in the complete Segal space model structure.
\end{definition}

\begin{remark}
  For readers familiar with the notions of hammock localization and $DK$-equivalence \cite{dwyer-kan}, Barwick and Kan have shown in \cite[1.8]{barwick-kan} that a relative functor
  \[
  F \cn (\C,\cW) \to (\C',\cW')
  \]
  is an equivalence of homotopy theories if and only if it induces a $DK$-equivalence between hammock localizations.
  In that case, $F$ induces equivalences between mapping simplicial sets and between categories of components.
  In particular, if $F$ is an equivalence of homotopy theories then the induced functor between categorical localizations
  \[
  \C[\cW^\inv] \to \C'[(\cW')^\inv]
  \]
  is an equivalence.
\end{remark}

\begin{proposition}[{\cite[2.8]{gjo1}}]\label{gjo28}
  Suppose
  \[
  F \cn (\C,\cW) \to (\C',\cW')
  \]
  is a relative functor and suppose that $F$ induces a weak equivalence of simplicial sets
  \begin{equation}\label{eq:F-level-we}
  \Ner\big((\C,\cW)^{\De[n]}\big) \to \Ner\big((\C',\cW')^{\De[n]}\big)
  \end{equation}
  for each $n$.  Then $F$ is an equivalence of homotopy theories.
\end{proposition}
\begin{proof}
  The assumption that \eqref{eq:F-level-we} is a weak equivalence for each $n$ means that
  \[
  \Ncl F\cn \Ncl(\C,\cW) \to \Ncl(\C',\cW')
  \]
  is a weak equivalence between classification diagrams in the Reedy model structure \cite[Section~2.4]{rezk-homotopy-theory}.
  Thus $\Ncl F$ is a weak equivalence in the complete Segal space model structure because it is a localization of the Reedy model structure.
  As a consequence, $\Ncl F$ induces a weak equivalence between the homotopy theories given by fibrant replacements.
\end{proof}

The following application of \cref{gjo28} is a special case of \cite[2.9]{gjo1} that is suitable for our purposes.
\begin{proposition}\label{gjo29}
  Suppose given relative functors
  \[
  F \cn (\C, \cW) \lradj (\C',\cW') \cn E
  \]
  such that
  \begin{itemize}
  \item $F$ is left adjoint to $E$ and
  \item each component of the unit, respectively counit, for $F \dashv E$ is a morphism in $\cW$, respectively $\cW'$.
  \end{itemize}
  Then $F$ and $E$ are equivalences of homotopy theories.
\end{proposition}
\begin{proof}
  The adjunction $F \dashv E$ induces an adjunction
  \[
  F^{\De[n]} \cn \C^{\De[n]} \lradj \C'^{\De[n]} \cn E^{\De[n]}
  \]
  for each $n$.
  Since the components of the unit and counit are morphisms in $\cW$ and $\cW'$, respectively, they continue to induce an adjunction when $F$ and $E$ are restricted to the subcategories $(\C,\cW)^{\De[n]}$ and $(\C',\cW')^{\De[n]}$.
  A natural transformation between functors induces a simplicial homotopy on nerves, and hence the result follows from \cref{gjo28}.
\end{proof}

\section{Permutative Categories}\label{sec:permcat}

In this section we define permutative categories, strict monoidal functors, and monoidal natural transformations.  See \cite{joyal-street,maclane,johnson-yau,cerberusI,cerberusII} for further discussion in the more general context of plain/braided/symmetric monoidal structure.

\begin{definition}\label{definition:permutativecategory}
  A \emph{permutative category} $(\C,\oplus,e,\xi)$ consists of
  \begin{itemize}
  \item a category $\C$,
  \item a functor $\oplus \cn \C \times \C \to \C$, called the \emph{monoidal sum},
  \item an object $e \in \C$, called the \emph{monoidal unit}, and
  \item a natural isomorphism $\xi$ called the \emph{symmetry isomorphism} with components
    \[
    \xi_{X,Y} \cn X \oplus Y \to Y \oplus X
    \]
    for objects $X$ and $Y$ of $\C$.
  \end{itemize}
  The monoidal sum is required to be associative and unital, with $e$ as its unit.
  The symmetry isomorphism $\xi$ is required to make the following symmetry and hexagon diagrams commute for objects $X,Y,Z \in \C$. 
  \begin{equation}\label{symmoncatsymhexagon}
    \begin{tikzpicture}[xscale=3,yscale=1.1,vcenter]
      \tikzset{0cell/.append style={nodes={scale=.85}}}
      \tikzset{1cell/.append style={nodes={scale=.85}}}
      \def\h{.3}
      \draw[0cell] 
      (0,0) node (a) {X \oplus Y}
      (a)++(.8,0) node (c) {X \oplus Y}
      (a)++(.4,-1) node (b) {Y \oplus X}
      ;
      \draw[1cell] 
      (a) edge node {1_{X \oplus Y}} (c)
      (a) edge node [swap,pos=.3] {\xi_{X,Y}} (b)
      (b) edge node [swap,pos=.7] {\xi_{Y,X}} (c)
      ;
      \begin{scope}[shift={(2,.5)}]
        \draw[0cell] 
        (0,0) node (x11) {(Y \oplus X) \oplus Z}
        (x11)++(.7,0) node (x12) {Y \oplus (X \oplus Z)}
        (x11)++(-110:1.3) node (x21) {(X \oplus Y) \oplus Z}
        (x12)++(-70:1.3) node (x22) {Y \oplus (Z \oplus X)}
        (x21)++(-70:.7) node (x31) {X \oplus (Y \oplus Z)}
        (x22)++(-110:.7) node (x32) {(Y \oplus Z) \oplus X}
        ;
        \draw[1cell]
        (x21) edge node[pos=.25] {\xi_{X,Y} \oplus 1_Z} (x11)
        (x11) edge[equal] node {} (x12)
        (x12) edge node[pos=.75] {1_Y \oplus \xi_{X,Z}} (x22)
        (x21) edge[equal] node[swap,pos=.25] {} (x31)
        (x31) edge node {\xi_{X,Y \oplus Z}} (x32)
        (x32) edge[equal] node[swap,pos=.75] {} (x22)
        ;
      \end{scope}
    \end{tikzpicture}
  \end{equation}
  A permutative category is also called a \emph{strict symmetric monoidal category}.
  The strictness refers to the conditions that the monoidal sum be strictly associative and unital.
\end{definition}

\begin{definition}\label{definition:strictmonoidalfunctor}
  Suppose $\C$ and $\D$ are permutative categories.  A \emph{symmetric monoidal functor}
  \[
  (P,P^2,P^0) \cn \C \to \D
  \]
  consists of a functor $P \cn \C \to \D$ together with natural transformations
  \[
    PX \oplus PY \fto{P^2} P(X \oplus Y) \andspace
    e \fto{P^0} Pe
  \]
  for objects $X,Y \in \C$, called the \emph{monoidal constraint} and \emph{unit constraint}, respectively.  These data satisfy the following associativity, unity, and symmetry axioms.
  \begin{description}
  \item[Associativity] The following diagram is commutative for all objects $X,Y,Z \in \C$.
    \begin{equation}\label{eq:sm-monoidal}
      \begin{tikzpicture}[xscale=4.2,yscale=1.5,vcenter]
        \tikzset{0cell/.append style={nodes={scale=.85}}}
        \tikzset{1cell/.append style={nodes={scale=.85}}}
        \draw[0cell] 
        (.2,0) node (a0) {
          \big(PX \oplus PY\big) \oplus PZ
        }
        (.8,0) node (b0) {
          PX \oplus \big(PY \oplus PZ\big)
        }
        (0,-1) node (a1) {
          P(X \oplus Y) \oplus PZ
        }
        (1,-1) node (b1) {
          PX \oplus P(Y \oplus Z)
        }
        (.2,-2) node (a2) {
          P\big((X \oplus Y) \oplus Z\big)
        }
        (.8,-2) node (b2) {
          P\big(X \oplus (Y \oplus Z)\big)
        }
        ;
        \draw[1cell] 
        (a0) edge[equal] node {} (b0)
        (a2) edge[equal] node {} (b2)
        (a0) edge['] node {P^2 \oplus 1_{PZ}} (a1)
        (a1) edge['] node[pos=.3] {P^2} (a2)
        (b0) edge node {1_{PX} \oplus P^2} (b1)
        (b1) edge node[pos=.3] {P^2} (b2)
        ;
      \end{tikzpicture}
    \end{equation}

  \item[Unity] The following two diagrams are commutative for all objects $X \in \C$.
    \begin{equation}\label{eq:sm-unit}
      \begin{tikzpicture}[xscale=2.25,yscale=1.5,vcenter]
        \tikzset{0cell/.append style={nodes={scale=.85}}}
        \tikzset{1cell/.append style={nodes={scale=.85}}}
        \draw[0cell] 
        (.5,-.3) node (a) {e \oplus PX}
        (a)++(.5,0) node (b) {PX}
        (0,-1) node (c) {Pe \oplus PX}
        (1,-1) node (d) {P(e \oplus X)}
        ;
        \draw[1cell] 
        (a) edge[equal] node {} (b)
        (c) edge node {P^2} (d)
        (a) edge['] node {P^0 \oplus 1_{PX}} (c)
        (b) edge[equal] node {} (d)
        ;
      \end{tikzpicture}
      \andspace
      \begin{tikzpicture}[xscale=2.25,yscale=1.5,vcenter]
        \tikzset{0cell/.append style={nodes={scale=.85}}}
        \tikzset{1cell/.append style={nodes={scale=.85}}}
        \draw[0cell] 
        (.5,-.3) node (a) {PX \oplus e}
        (a)++(.5,0) node (b) {PX}
        (0,-1) node (c) {PX \oplus Pe}
        (1,-1) node (d) {P(X \oplus e)}
        ;
        \draw[1cell] 
        (a) edge[equal] node {} (b)
        (c) edge node {P^2} (d)
        (a) edge['] node {1_{PX} \oplus P^0} (c)
        (b) edge[equal] node {} (d)
        ;
      \end{tikzpicture}
    \end{equation}
  \item[Symmetry] The following diagram is commutative for all objects $X,Y \in \C$.
      \begin{equation}\label{eq:sm-symm}
        \begin{tikzpicture}[xscale=4,yscale=1.1,vcenter]
          \tikzset{0cell/.append style={nodes={scale=.85}}}
          \tikzset{1cell/.append style={nodes={scale=.85}}}
          \draw[0cell] 
          (0,0) node (a0) {PX \oplus PY}
          (1,0) node (b0) {PY \oplus PX}
          (0,-1) node (a1) {P(X \oplus Y)}
          (1,-1) node (b1) {P(Y \oplus X)}
          ;
          \draw[1cell] 
          (a0) edge node {\xi_{PX,PY}} (b0)
          (a1) edge node {P\xi_{X,Y}} (b1)
          (a0) edge['] node {P^2} (a1)
          (b0) edge node {P^2} (b1)
          ;
        \end{tikzpicture}
      \end{equation}
    \end{description}
    This finishes the definition of a symmetric monoidal functor.
    A \emph{strict symmetric monoidal functor} is a symmetric monoidal functor $P$ such that the monoidal constraint $P^2$ and the unit constraint $P^0$ are both identity natural transformations.
\end{definition}

\begin{definition}\label{definition:monoidal-nt}
  Suppose $P,Q\cn \C \to \D$ are strict symmetric monoidal functors between permutative categories.
  A \emph{monoidal natural transformation}
  \[
  \alpha\cn P \to Q
  \]
  is a natural transformation between the underlying functors such that the
  following monoidal and unit conditions hold:
  \begin{equation}\label{eq:monoidal-nt}
  \alpha_{X \oplus Y} = \alpha_X \oplus \alpha_Y
  \andspace
  \alpha_{e} = 1_{e}
  \end{equation}
  for all objects $X$ and $Y$ in $\C$.
\end{definition}

\begin{definition}\label{definition:permcat}
  We let $\PermCat$ denote the 2-category of small permutative categories, strict symmetric monoidal functors, and monoidal natural transformations.  Identities and compositions are given by those of the underlying functors and natural transformations.
\end{definition}
\begin{remark}
  There are several weaker variants for symmetric monoidal structure both on categories and on functors thereof.
  Thomason \cite[1.9.2]{thomason} shows that each of the standard variants have equivalent stable homotopy theories.
  See \cite{gjo-extending} for a general treatment via 2-dimensional monad theory.

  Here we restrict to permutative categories and strict symmetric monoidal functors because they have the most direct comparison with multicategories.
  For more on this point, see \cref{remark:epz-2nat-strictness} below.
\end{remark}

\section{Multicategories}\label{sec:multicat}

In this section we review the definitions of multicategories, multifunctors, and multinatural transformations.
Several of the details here will be needed for our explanation of the free functor in \cref{sec:free-perm,sec:free-perm-adj} below.
For further details and context we refer the reader to \cite[Chapter 5]{cerberusIII} and \cite{yau-operad}.

\begin{definition}\label{def:profile}
Suppose $C$ is a class.  
\begin{itemize}
\item Denote by
\[\Prof(C) = \coprodover{n \geq 0}\ C^{\times n}\] 
the class of finite ordered sequences of elements in $C$.  An element in $\Prof(C)$ is called a \emph{$C$-profile}.  
\item A typical $C$-profile of length
  $n=\len\angc$ is denoted by $\angc = (c_1, \ldots, c_n) \in
  C^{n}$ or by $\ang{c_i}_i$ to indicate
  the indexing variable.  The empty $C$-profile
  is denoted by $\ang{}$.
\item We let $\oplus$ denote the concatenation of profiles, and note
  that $\oplus$ is an associative binary operation with unit given by
  the empty tuple $\ang{}$.
\item An element in $\Prof(C)\times C$ is denoted
  as $\IMMduc$ with $c'\in C$ and
  $\angc\in\Prof(C)$.
  \dqed 
\end{itemize}
\end{definition}

\begin{definition}\label{def:multicategory}
A \emph{multicategory} $(\M, \gamma, \operadunit)$ consists of the following data.
\begin{itemize}
\item $\M$ is equipped with a class $\ObM$ of \emph{objects}.  We write $\Prof(\M)$ for $\Prof(\Ob\M)$.
\item For $c'\in\ObM$ and $\angc=(c_1,\ldots,c_n)\in\ProfM$, $\M$ is equipped with a set of \emph{$n$-ary operations}
  \[\M\IMMduc = \M\mmap{c'; c_1,\ldots,c_n}\]
  with \emph{input profile} $\angc$ and \emph{output} $c'$.
\item For $\IMMduc \in \ProfMM$ as above and a permutation $\sigma \in
  \Sigma_n$, $\M$ is equipped with an isomorphism of sets
  \[\begin{tikzcd}\M\IMMduc \rar{\sigma}[swap]{\cong} & \M\IMMducsigma,\end{tikzcd}\]
  called the \emph{right action} or the \emph{symmetric group action}, in which
  \[\angc\sigma = (c_{\sigma(1)}, \ldots, c_{\sigma(n)})\]
  is the right permutation of $\angc$ by $\sigma$.
\item For $c \in \ObM$, $\M$ is equipped with an element
  \[\operadunit_c \in \M\IMMcc,\]
  called the \emph{$c$-colored unit}.
\item For
  $c'' \in \ObM$, $\ang{c'} = (c'_1,\ldots,c'_n) \in \ProfM$, and $\ang{c_j} = (c_{j,1},\ldots,c_{j,k_j}) \in \ProfM$ for each $j\in\{1,\ldots,n\}$, let $\angc = \oplus_j\ang{c_j} \in \ProfM$ be the concatenation of the $\ang{c_j}$.
  Then $\M$ is equipped with a map
  \begin{equation}\label{eq:defn-gamma}
    \begin{tikzcd}
      \M\mmap{c'';\ang{c'}} \times
      \prod\limits_{j=1}^n \M\mmap{c_j';\ang{c_j}}
      \rar{\gamma}
      &
      \M\mmap{c'';\ang{c}}
    \end{tikzcd}
  \end{equation}
  called the \emph{composition} or \emph{multicategorical composition}. 
\end{itemize}
These data are required to satisfy the following axioms.
\begin{description}
\item[Symmetric Group Action]
For $\mmap{c';\ang{c}}\in\ProfMM$ with $n=\len\ang{c}$ and
$\sigma,\tau\in\Sigma_n$, the following diagram commutes. 
\begin{equation}\label{multicategory-symmetry}
\begin{tikzcd}
\M\IMMduc \arrow{rd}[swap]{\sigma\tau} \rar{\sigma} & \M\IMMducsigma \dar{\tau}\\
& \M\mmap{c';\angc\sigma\tau}
\end{tikzcd}
\end{equation}
Moreover, the identity permutation in $\Sigma_n$ acts as the identity map on $\M\IMMduc$.
\item[Associativity]
Suppose given
\begin{itemize}
\item $c''' \in \ObM$,
\item $\ang{c''} = (c''_1,\ldots,c''_{n}) \in \ProfM$,
\item $\ang{c_j'} = (c'_{j,1},\ldots,c'_{j,k_{j}}) \in \ProfM$ for each
  $j \in \{1,\ldots,n\}$, and
\item $\ang{c_{j,i}} = (c_{j,i,1},\ldots,c_{j,i,\ell_{j,i}}) \in \ProfM$ for
  each $j\in\{1,\ldots,n\}$ and each $i \in \{1,\ldots,k_j\}$,
\end{itemize}
such that $k_j = \len\ang{c_j'} > 0$ for at least one $j$.  For each $j$,
let $\ang{c_j} = \oplus_{i=1}^{k_j}{\ang{c_{j,i}}}$ denote the concatenation of
the $\ang{c_{j,i}}$.  Let $\ang{c} =
\oplus_{j=1}^n{\ang{c_{j}}}$ denote the concatenation of the $\ang{c_j}$.  Then the \emph{associativity diagram} below commutes.
\begin{equation}\label{multicategory-associativity}
\begin{tikzpicture}[x=40mm,y=17mm,vcenter]
  \draw[0cell=.85] 
  (0,0) node (a) {\textstyle
    \M\mmap{c''';\ang{c''}}
    \times
    \biggl[\prod\limits_{j=1}^n \M\mmap{c''_j;\ang{c'_{j}}}\biggr]
    \times
    \prod\limits_{j=1}^n \biggl[\prod\limits_{i=1}^{k_j} \M\mmap{c'_{j,i};\ang{c_{j,i}}}\biggr] 
  }
  (1,.8) node (b) {\textstyle
    \M\mmap{c''';\ang{c'}}
    \times
    \prod\limits_{j=1}^{n} \biggl[\prod\limits_{i=1}^{k_j} \M\mmap{c'_{j,i};\ang{c_{j,i}}}\biggr]
  }
  (0,-1.2) node (a') {\textstyle
    \M\mmap{c''';\ang{c''}} \times
    \prod\limits_{j=1}^n \biggl[\M\mmap{c_j'';\ang{c_j'}} \times \prod\limits_{i=1}^{k_j} \M\mmap{c'_{j,i};\ang{c_{j,i}}}\biggr]
  }
  (1,-2) node (b') {\textstyle
    \M\mmap{c''';\ang{c''}} \times \prod\limits_{j=1}^n \M\mmap{c_j'';\ang{c_{j}}}
  }
  (1.2,-.6) node (c) {\textstyle
    \M\mmap{c''';\ang{c}}
  }
  ;
  \draw[1cell=.85]
  (a) edge node {\iso} node['] {\mathrm{permute}} (a')
  (a) edge[shorten >=-2ex] node[pos=.3] {(\ga,1)} (b)
  (b) edge node {\ga} (c)
  (a') edge['] node {(1,\textstyle\prod_j \ga)} (b')
  (b') edge['] node {\ga} (c)
  ;
\end{tikzpicture}
\end{equation}

\item[Unity]
Suppose $c' \in \ObM$.
\begin{enumerate}
\item If $\angc = (c_1,\ldots,c_n) \in \ProfM$ has length $n \geq 1$, then the following \emph{right unity diagram} is commutative.
  Here $\boldone$ is the one-point set and $\boldone^n$ is its $n$-fold Cartesian product.
  \begin{equation}\label{enr-multicategory-right-unity}
    \begin{tikzcd} \M\IMMduc \times \boldone^{n} \dar[swap]{1 \times (\times_j \operadunit_{c_j})} \rar{\rho} & \M\IMMduc \dar{1}\\
      \M\IMMduc \times \prod\limits_{j=1}^n \M\IMMcjcj \rar{\gamma} & \M\IMMduc
    \end{tikzcd}
  \end{equation}

\item
For any $\ang{c} \in \ProfM$, the \emph{left unity diagram}
below is commutative.
\begin{equation}\label{enr-multicategory-left-unity}
\begin{tikzcd}
\boldone \times \M\IMMduc \dar[swap]{\operadunit_{c'} \times 1} \rar{\lambda} & 
\M\IMMduc \dar{1}\\
\M\mmap{c';c'} \times \M\IMMduc \rar{\gamma} & \M\IMMduc
\end{tikzcd}
\end{equation}
\end{enumerate}
\item[Equivariance]
Suppose that in the definition of $\gamma$ \eqref{eq:defn-gamma}, $\mathrm{len}\ang{c_j} = k_j \geq 0$.
\begin{enumerate}
\item For each $\sigma \in \Sigma_n$, the following \emph{top equivariance diagram} is commutative.
\begin{equation}\label{enr-operadic-eq-1}
\begin{tikzcd}[column sep=large,cells={nodes={scale=.9}},
every label/.append style={scale=.9}]
\M\mmap{c'';\ang{c'}} \times \prod\limits_{j=1}^n \M\mmap{c'_j;\ang{c_j}} 
\dar[swap]{\gamma} \rar{(\sigma, \sigma^{-1})}
& \M\mmap{c'';\ang{c'}\sigma} \times \prod\limits_{j=1}^n \M\mmap{c'_{\sigma(j)};\ang{c_{\sigma(j)}}} \dar{\gamma}\\
\M\mmap{c'';\ang{c_1},\ldots,\ang{c_n}} \rar{\sigma\langle k_{\sigma(1)}, \ldots , k_{\sigma(n)}\rangle}
& \M\mmap{c'';\ang{c_{\sigma(1)}},\ldots,\ang{c_{\sigma(n)}}}
\end{tikzcd}
\end{equation}
Here $\sigma\langle k_{\sigma(1)}, \ldots , k_{\sigma(n)} \rangle \in \Sigma_{k_1+\cdots+k_n}$ is right action of the block permutation  that permutes the $n$ consecutive blocks of lengths $k_{\sigma(1)}$, $\ldots$, $k_{\sigma(n)}$ as $\sigma$ permutes $\{1,\ldots,n\}$, leaving the relative order within each block unchanged.
\item
Given permutations $\tau_j \in \Sigma_{k_j}$ for $1 \leq j \leq n$,
the following \emph{bottom equivariance
  diagram} is commutative.
\begin{equation}\label{enr-operadic-eq-2}
\begin{tikzcd}[cells={nodes={scale=.9}},
every label/.append style={scale=.9}]
\M\mmap{c'';\ang{c'}} \times \prod\limits_{j=1}^n \M\mmap{c'_j;\ang{c_j}}
\dar[swap]{\gamma} \rar{(1, \times_j \tau_j)} & 
\M\mmap{c'';\ang{c'}} \times \prod\limits_{j=1}^n \M\mmap{c'_j;\ang{c_j}\tau_j}\dar{\gamma} \\
\M\mmap{c'';\ang{c_1},\ldots,\ang{c_n}} \rar{\tau_1 \times \cdots \times \tau_n}
& \M\mmap{c'';\ang{c_1}\tau_1,\ldots,\ang{c_n}\tau_n}
\end{tikzcd}
\end{equation}
Here the block sum $\tau_1 \times\cdots \times\tau_n \in \Sigma_{k_1+\cdots+k_n}$ is the image of $(\tau_1, \ldots, \tau_n)$ under the canonical inclusion \[\Sigma_{k_1} \times \cdots \times \Sigma_{k_n} \to \Sigma_{k_1 + \cdots + k_n}.\]
\end{enumerate}
\end{description}
This finishes the definition of a multicategory.  

Moreover, we define the following.
\begin{itemize}
\item A multicategory is \emph{small} if its class of objects is a set.
\item An \emph{operad} is a multicategory with one object.  If $\M$ is an operad, then its set of $n$-ary operations is denoted by $\M_n$.
\item The \emph{initial operad} $\Mtu$ consists of a single object $*$ and its unit operation.
\item The \emph{terminal multicategory} $\Mterm$ consists of a single object $*$ and a single $n$-ary operation $\iota_n$ for each $n \ge 0$.\dqed
\end{itemize}
\end{definition}

\begin{example}[Endomorphism Operad]\label{example:End}
  Suppose $\M$ is a multicategory and $c$ is an object
  of $\M$.  Then $\End(c)$ is the operad consisting of
  the single object $c$ and $n$-ary operation object
  \[\End(c)_n = \M\mmap{c;\ang{c}},\]
  where $\ang{c}$ denotes the constant $n$-tuple at $c$.  The
  symmetric group action, unit, and composition of $\End(c)$ are given
  by those of $\M$.
\end{example}

\begin{example}[Endomorphism Multicategory]\label{ex:endc}
  Suppose $(\C,\oplus,e,\xi)$ is a small permutative category.  Then the
  \emph{endomorphism multicategory} $\End(\C)$ is the small multicategory with object set $\Ob\C$ and with
  \[\End(\C)\mmap{Y;\ang{X}} = \C(X_1 \oplus \cdots \oplus
  X_n , Y)\]
  for $Y \in \Ob\C$ and $\ang{X}=(X_1,\cdots,X_n) \in (\Ob\C)^{\times n}$.  An empty $\oplus$ means the unit object $e$.
\end{example}

\begin{example}[Underlying Category]\label{ex:unarycategory}
Each multicategory $(\M,\ga,\operadunit)$ has an underlying category with
\begin{itemize}
\item the same objects,
\item identities given by the colored units, and
\item composition given by 
\[\begin{tikzcd}[column sep=large]
\M\scmap{b;c} \times \M\scmap{a;b} \ar{r}{\gamma} & \M\scmap{a;c}
\end{tikzcd}\]
for objects $a, b, c \in \M$.\dqed
\end{itemize}   
\end{example}

\subsection*{The 2-Category of Small Multicategories}

\begin{definition}\label{def:multicategory-functor}
A \emph{multifunctor} $F \cn \M \to \N$ between multicategories $\M$ and $\N$ consists of
\begin{itemize}
\item an object assignment $F \cn \ObM \to \ObN$ and
\item for each $\mmap{c';\ang{c}} \in \ProfMM$ with $\angc=(c_1,\ldots,c_n)$, a component morphism 
\[F \cn \M\mmap{c';\ang{c}} \to \N\mmap{Fc';F\ang{c}},\] where $F\angc=(Fc_1,\ldots,Fc_n)$.
\end{itemize}
These data are required to preserve the symmetric group action, the colored units, and the composition in the following sense.
\begin{description}
\item[Symmetric Group Action] For each $\IMMduc$ as above and each
  permutation $\sigma \in \Sigma_n$, the following
  diagram is commutative.
\begin{equation}\label{multifunctor-equivariance}
\begin{tikzcd}
\M\mmap{c';\ang{c}} \ar{d}{\cong}[swap]{\sigma} \ar{r}{F} & \N\mmap{Fc';F\ang{c}} \ar{d}{\cong}[swap]{\sigma}\\
\M\mmap{c';\ang{c}\sigma} \ar{r}{F} & \N\mmap{c';F\ang{c}\sigma}\end{tikzcd}
\end{equation}
\item[Units] For each $c\in\ObM$, the following diagram is commutative.
\begin{equation}\label{multifunctor-unit}
\begin{tikzpicture}[x=25mm,y=15mm,vcenter]
  \draw[0cell] 
  (0,0) node (a) {\boldone}
  (1,.5) node (b) {\M\mmap{c;c}}
  (1,-.5) node (b') {\N\mmap{Fc;Fc}}
  ;
  \draw[1cell] 
  (a) edge node {\operadunit_c} (b)
  (a) edge node[swap,pos=.6] {\operadunit_{Fc}} (b')
  (b) edge node {F} (b')
  ;
\end{tikzpicture}
\end{equation} 
\item[Composition] For $c''$, $\ang{c'}$, and $\ang{c} =
  \oplus_j\ang{c_j}$ as in the definition of $\gamma$
  \eqref{eq:defn-gamma}, the following diagram is commutative.
\begin{equation}\label{v-multifunctor-composition}
\begin{tikzcd}[column sep=large,cells={nodes={scale=.9}},
every label/.append style={scale=.9}]
\M\mmap{c'';\ang{c'}} \times \prod\limits_{j=1}^n \M\mmap{c'_j;\ang{c_j}} \dar[swap]{\gamma} \ar{r}{(F,\prod_j F)} & \N\mmap{Fc'';F\ang{c'}} \times \prod\limits_{j=1}^n \N\mmap{Fc'_j;F\ang{c_j}} \dar{\gamma}\\  
\M\mmap{c'';\ang{c}} \ar{r}{F} & \N\mmap{Fc'';F\ang{c}}
\end{tikzcd}
\end{equation}
\end{description}
This finishes the definition of a multifunctor.  

Moreover, we define the following.
\begin{itemize}
\item A multifunctor $\M \to \N$ is also called an \emph{$\M$-algebra in $\N$}.
\item For another multifunctor $G \cn \N\to\P$ between multicategories, where $\P$ has object class $\ObP$, the \emph{composition} $GF \cn \M\to\P$ is the multifunctor defined by composing the assignments on objects 
\[\begin{tikzcd} \ObM \ar{r}{F} & \ObN \ar{r}{G} & \ObP
\end{tikzcd}\]
and the morphisms on $n$-ary operations
\[\begin{tikzcd}
\M\mmap{c';\ang{c}} \ar{r}{F} & \N\mmap{Fc';F\ang{c}} \ar{r}{G} & \P\mmap{GFc';GF\ang{c}}.
\end{tikzcd}\]
\item The \emph{identity multifunctor} $1_{\M} \cn \M\to\M$ is defined by the identity assignment on objects and the identity morphism on $n$-ary operations.
\item An \emph{operad morphism} is a multifunctor between two multicategories with one object.\dqed
\end{itemize}
\end{definition}

\begin{definition}\label{def:multicat-natural-transformation}
Suppose $F,G \cn \M\to\N$ are multifunctors as in \cref{def:multicategory-functor}.
A \emph{multinatural transformation} $\theta \cn F\to G$ consists of component morphisms
\[\theta_c \in \N\mmap{Gc;Fc} \forspace c\in\ObM\]
such that the following \emph{naturality diagram} commutes for each $\mmap{c';\ang{c}} \in \ProfMM$ with $\angc=(c_1,\ldots,c_n)$.
\begin{equation}\label{multinat}
\begin{tikzpicture}[x=25mm,y=12mm,vcenter]
  \draw[0cell=.85]
  (.5,0) node (a) {\M\mmap{c';\ang{c}}}
  (1,1) node (b) {\boldone \times \M\mmap{c';\ang{c}}}
  (3,1) node (c) {\N\mmap{Gc';Fc'} \times \N\mmap{Fc';F\ang{c}}}
  (3.5,0) node (d) {\N\mmap{Gc';F\ang{c}}}
  (1,-1) node (b') {\M\mmap{c';\ang{c}}\times \prod\limits_{j=1}^n \boldone}
  (3,-1) node (c') {\N\mmap{Gc';G\ang{c}} \times \prod\limits_{j=1}^n \N\mmap{Gc_j;Fc_j}}
  ;
  \draw[1cell=.85] 
  (a) edge node[pos=.3] {\la^\inv} (b)
  (a) edge node[swap,pos=.2] {\rho^\inv} (b')
  (b) edge node {\theta_{c'} \times F} (c)
  (c) edge node[pos=.7] {\ga} (d)
  (b') edge node {G \times \prod\limits_{j=1}^n \theta_{c_j}} (c')
  (c') edge['] node[pos=.7] {\ga} (d)
  ;
\end{tikzpicture}
\end{equation}
This finishes the definition of a multinatural transformation.  

Moreover, we define the following.
\begin{itemize}
\item The \emph{identity multinatural transformation} $1_F \cn F\to F$ has components \[(1_F)_c = \operadunit_{Fc} \forspace c\in\ObM.\dqed\]
\end{itemize} 
\end{definition}

\begin{definition}\label{def:multinatural-composition}
Suppose $\theta \cn F \to G$ is a multinatural transformation between multifunctors as in \cref{def:multicat-natural-transformation}.
\begin{enumerate}
\item Suppose $\beta \cn G \to H$ is a multinatural transformation for a multifunctor $H \cn \M \to \N$.  The \emph{vertical composition}
\begin{equation}\label{multinatvcomp}
  \beta\theta \cn F \to H
\end{equation} 
is the multinatural transformation with components at $c \in \ObM$ given by the following composites.
  \[
  \begin{tikzpicture}[x=45mm,y=15mm]
    \draw[0cell] 
    (0,0) node (a) {\boldone}
    (0,-1) node (b) {\boldone \times \boldone}
    (1,-1) node (c) {\N\mmap{Hc;Gc} \times \N\mmap{Gc;Fc}}
    (1,0) node (d) {\N\mmap{Hc;Fc}}
    ;
    \draw[1cell] 
    (a) edge node['] {\la^\inv} (b)
    (b) edge node {\be_c \times \theta_c} (c)
    (c) edge['] node {\ga} (d)
    (a) edge node {(\be\theta)_c} (d)
    ;
  \end{tikzpicture}
  \]
\item Suppose $\theta' \cn F' \to G'$ is a multinatural transformation for multifunctors $F', G' \cn \N \to \P$.  The \emph{horizontal composition}
\begin{equation}\label{multinathcomp}
\theta' \ast \theta \cn F'F \to G'G
\end{equation} 
is the multinatural transformation with components at $c \in \ObM$ given by the following composites.
\[
\begin{tikzpicture}[x=50mm,y=12mm]
  \draw[0cell=.9] 
  (0,0) node (a) {\boldone}
  (0,-2) node (b) {\boldone \times \boldone}
  (1,-2) node (c) {\P\mmap{G'Gc;F'Gc} \times \N\mmap{Gc;Fc}}
  (1,-1) node (d) {\P\mmap{G'Gc;F'Gc} \times \P\mmap{F'Gc;F'Fc}}
  (1,0) node (e) {\P\mmap{G'Gc;F'Fc}}
  ;
  \draw[1cell=.9] 
  (a) edge node {(\theta' * \theta)_c} (e)
  (a) edge['] node {\la^\inv} (b)
  (b) edge node {\theta'_{Gc} \times \theta_c} (c)
  (c) edge['] node {1 \times F'} (d)
  (d) edge['] node {\ga} (e)
  ;
\end{tikzpicture}
\]
\end{enumerate}
This finishes the definition.
\end{definition}

\begin{theorem}[{\cite[2.4.26]{johnson-yau}}]\label{multicat-2cat}
There is a 2-category $\Multicat$ consisting of the following data.
\begin{itemize}
\item Its objects are small multicategories.
\item For small multicategories $\M$ and $\N$, the hom category 
\[\Multicat(\M,\N)\]
has
\begin{itemize}
\item multifunctors $\M\to\N$ as 1-cells,
\item multinatural transformations as 2-cells,
\item vertical composition as composition, and
\item identity multinatural transformations as identity 2-cells.
\end{itemize}
\item The identity 1-cell $1_{\M}$ is the identity multifunctor $1_{\M}$.
\item Horizontal composition of 1-cells is the composition of multifunctors.
\item Horizontal composition of 2-cells is that of multinatural transformations.
\end{itemize}
\end{theorem}

Recall from \cref{ex:endc} the endomorphism multicategory $\End(\C)$ of a permutative category $\C$.  A strict symmetric monoidal functor $F\cn \C \to \D$ between permutative categories induces a multifunctor
\[
\End(F)\cn \End(\C) \to \End(\D)
\]
by applying $F$ to the operations of $\End(\C)$, which are morphisms in $\C$.  Similarly, a monoidal natural transformation $\theta\cn F \to G$ between strict symmetric monoidal functors induces a multinatural transformation
\[
\End(\theta)\cn \End(F) \to \End(G)
\]
whose components are given by those of $\theta$.  See \cite[Section 5.3]{cerberusIII} for further explanation of these constructions.
\begin{proposition}[{\cite[5.3.6]{cerberusIII}}]\label{proposition:end-2fun}
  The construction $\End(\C)$ of \cref{ex:endc} defines a 2-functor
  \[
  \End\cn \PermCat \to \Multicat
  \]
  from the 2-category of small permutative categories to the 2-category of small multicategories.
\end{proposition}
\begin{remark}\label{remark:permcat}
  In \cite{cerberusIII} $\PermCat$ denotes the larger 2-category of small permutative categories and general (not-necessarily-strict) symmetric monoidal functors.  The proof of \cite[5.3.6]{cerberusIII} shows that $\End$ is 2-functorial with respect to these more general 1-cells.  The statement given here is restricted to the subcategory consisting of strict symmetric monoidal functors.  In \cref{remark:epz-2nat-strictness} we note one reason for this restriction in the present work.
\end{remark}

\section{Free Permutative Category on a Multicategory}\label{sec:free-perm}

In this section we describe a 2-functor
\[
F \cn \Multicat \to \permcat
\]
that is left 2-adjoint to the endomorphism 2-functor, $\End$.
Our description follows that of \cite[Theorem~4.2]{elmendorf-mandell-perm}, but we include full details here because this free construction is crucial to our main result, \cref{theorem:multicat-perm-hty-thy}.
The definition makes use of sequences $\ang{x}$, indexing functions $f$, and permutations $\si^k_{g,f}$.
We make the following preliminary definitions.
\begin{definition}\label{definition:free-perm-helper}
  For each natural number $r \ge 0$, let $\ufs{r}$ denote the finite set with $r$ elements
  \[
    \ufs{r} = \{1, \ldots, r\} \forspace r \ge 1, \andspace \ufs{0} = \varnothing.
  \]
  Suppose $\ang{x}$ is a sequence of length $r$, with each $x_i \in \M$.
  Suppose
  \[
  f\cn \ufs{r} \to \ufs{s} \andspace g\cn \ufs{s} \to \ufs{t}
  \]
  are functions of finite sets, for $r,s,t \ge 0$.
  Then we define the following.
  \begin{itemize}
  \item For $j \in \ufs{s}$, let
    \begin{equation}\label{eq:x-finv}
      \ang{x}_{f^\inv(j)} = \ang{x_i}_{i \in f^\inv(j)}
    \end{equation}
    denote the sequence formed by those $x_i$ with $i \in f^\inv(j)$, ordered as in $\ang{x}$.
    Similarly, for $k \in \ufs{t}$, let
    \[
    \ang{\phi}_{g^\inv(k)} = \ang{\phi_j}_{j \in g^\inv(k)},
    \]
    where $\ang{\phi}$ is a length-$s$ sequence of operations in $\M$.
  \item For $k \in \ufs{t}$, let $\si^k_{g,f} \in \Si_t$ be the unique permutation such that
    \begin{equation}\label{eq:sigma-kgf}
      \bigg[
      \bigoplus_{j \in g^\inv(k)} \ang{x}_{f^\inv(j)}
      \bigg]
      \cdot \sigma^k_{g,f} 
      =
      \ang{x}_{(gf)^\inv(k)},
    \end{equation}
    where the sequence on the left hand side is the concatenation of sequences in the order specified by $g^\inv(k)$.
    We will use the action of these permutations on both objects and operations.
    \dqed
  \end{itemize}
\end{definition}

\begin{definition}\label{definition:free-perm}
  Suppose $\M$ is a multicategory.  Define a permutative category $F\M$, called the \emph{free permutative category on $\M$}, as follows.
  \begin{description}
  \item[Objects] The objects of $F\M$ are given by the $(\Ob\M)$-profiles: finite ordered sequences $\ang{x} = (x_1,\ldots,x_r)$ of objects of $\M$, with $r \ge 0$.
  \item[Morphisms] Given sequences $\ang{x}$ and $\ang{y}$ with lengths $r$ and $s$, respectively, the morphisms from $\ang{x}$ to $\ang{y}$ in $F\M$ are given by pairs $(f,\ang{\phi})$ consisting of
    \begin{itemize}
    \item a function
      \[
      f \cn \ufs{r} \to \ufs{s}
      \]
      called the \emph{index map} and
    \item an ordered sequence of operations
      \[
      \ang{\phi} \withspace \phi_j \in \M\mmap{y_j;\ang{x_i}_{i \in f^\inv(j)}}
      \]
      for $j \in \ufs{s}$.
    \end{itemize}
    The identity morphism on $\ang{x}$ is given by $1_{\ufs{r}}$ and the tuple of unit operations $1_{x_i}$.

  \item[Composition] The composition of a pair of morphisms
    \[
    \ang{x} \fto{(f,\ang{\phi})} \ang{y} \fto{(g,\ang{\psi})} \ang{z}
    \]
    is the pair
    \begin{equation}\label{eq:FM-comp}
    \big(gf , \ang{\theta_k \cdot \si^k_{g,f}}_{k \in \ufs{t}} \big),
    \end{equation}
    where, for each $k \in \ufs{t}$,
    \begin{equation}\label{eq:thetak}
      \theta_k = \ga\lrscmap{\psi_k;\ang{\phi}_{g^\inv(k)}} \in \M\lrscmap{\;\bigoplus_{j \in g^\inv(k)} \ang{x}_{f^\inv(j)} ; z_k}.
    \end{equation}
    Note that the input profile for $\theta_k$ is the concatenation of $\ang{x}_{f^\inv(j)}$ for $j \in g^\inv(k)$.
    By definition \cref{eq:sigma-kgf}, the right action of $\si^k_{g,f}$ permutes this input profile to $\ang{x}_{(gf)^\inv(k)}$.
    We check that composition of morphisms is unital and associative in 
    \cref{proposition:free-perm} below.

  \item[Monoidal Sum]
    The monoidal sum 
    \[
      \oplus \cn F\M \times F\M \to F\M
    \]
    is given on objects by concatenation of sequences.
    The monoidal sum of morphisms
    \[
      (f,\ang{\phi}) \cn \ang{x} \to \ang{y}
      \andspace
      (f',\ang{\phi'}) \cn \ang{x'} \to \ang{y'},
    \]
    is the pair
    \[
      (f \oplus f', \ang{\phi} \oplus \ang{\phi'})
    \]
    where $f \oplus f'$ denotes the composite
    \[
      \ufs{r + r'} \iso \ufs{r} \bincoprod \ufs{r'} \fto{f \bincoprod f'} \ufs{s} \bincoprod \ufs{s'} \iso \ufs{s + s'}
    \]
    given by the disjoint union of $f$ with $f'$ and the canonical order-preserving isomorphisms.
    Functoriality of the monoidal sum follows because disjoint union of indexing functions preserves preimages and the operations in a composite \eqref{eq:FM-comp} are determined elementwise for the indexing set of the codomain.

  \item[Monoidal Unit]
    The monoidal unit is the empty sequence $\ang{}$.  The unit and associativity isomorphisms for $\oplus$ are identities. 

  \item[Symmetry] 
    The symmetry isomorphism for sequences $\ang{x}$ of length $r$ and $\ang{x'}$ of length $r'$ is 
    \[
      \xi_{\ang{x},\ang{x'}} = \big(\tau_{r,r'} , \ang{1}\big)
    \]
    where
    \[
      \tau_{r,r'} \cn
      \ufs{r + r'} \iso \ufs{r} \bincoprod \ufs{r'} \to \ufs{r'} \bincoprod \ufs{r}
      \iso \ufs{r' + r}
    \]
    is induced by the block-transposition of $\ufs{r}$ with $\ufs{r'}$, keeping the relative order within each block fixed.

    Concatenation of sequences is strictly associative and unital.
    The symmetry and hexagon axioms \eqref{symmoncatsymhexagon} follow from the corresponding equalities of block permutations.
  \end{description}
  This completes the definition of $F\M$ and the data for its permutative structure.
\end{definition}

\begin{proposition}\label{proposition:free-perm}
  In the context of \cref{definition:free-perm}, $F\M$ is a permutative category.
\end{proposition}
\begin{proof}
  Observe that $\si^k_{g,f}$ is the identity permutation if either $f$ or $g$ is an identity.
  This implies that composition is strictly unital.
  To verify associativity of composition, suppose given the following objects and morphisms in $F\M$:
  \[
  \ang{x} \fto{(f,\ang{\phi})}
  \ang{y} \fto{(g,\ang{\psi})}
  \ang{z} \fto{(h, \ang{\mu})}
  \ang{w}
  \]
  with index maps
  \[
  \ufs{r} \fto{f}
  \ufs{s} \fto{g}
  \ufs{t} \fto{h}
  \ufs{u}.
  \]

  The two composites to be compared are
  \begin{equation}\label{eq:free-perm-assoc}
    \big((h,\ang{\mu})(g,\ang{\psi})\big)(f,\ang{\phi})
    \andspace
    (h,\ang{\mu})\big((g,\ang{\psi})(f,\ang{\phi})\big).
  \end{equation}
  The index map of both composites is $hgf$.  For each $\ell \in \ufs{u}$, the $\ell$th operations of \cref{eq:free-perm-assoc} are
  \[
  \ga\lrscmap{\theta_\ell \cdot \si^\ell_{h,g}; \ang{\phi}_{(hg)^\inv(\ell)}} \cdot \si^\ell_{hg,f}
  \andspace
  \ga\lrscmap{\mu_\ell;\ang{\om_k \cdot \si^k_{g,f}}_{k \in h^\inv(\ell)}} \cdot \si^\ell_{h,gf},
  \]
  respectively, where
  \[
  \theta_{\ell} = \ga\lrscmap{\mu_\ell; \ang{\psi}_{h^\inv(\ell)}}
  \andspace
  \om_{k} = \ga\lrscmap{\psi_k;\ang{\phi}_{g^\inv(k)}}.
  \]
  To see that the two composites are equal, we first use the top equivariance diagram \eqref{enr-operadic-eq-1} with $\si = \si^\ell_{h,g}$.
  For each $j \in \ufs{s}$ the action of $\si$ gives
  \[
  \ang{\phi_{\si^\inv(j)}}_{j \in (hg)^\inv(\ell)}
  = \bigoplus_{k \in h^\inv(\ell)}\ang{\phi}_{g^\inv(k)}
  \]
  by \eqref{eq:sigma-kgf} for $h$ and $g$.
  Then by top equivariance, with $r_j = |f^\inv(j)|$, we have
  \[
  \ga\lrscmap{\theta_\ell \cdot \si^\ell_{h,g}; \ang{\phi}_{(gh)^\inv(\ell)}} \cdot \si^\ell_{hg,f}
  =
  \ga\lrscmap{\theta_\ell; \bigoplus_{k \in h^\inv(\ell)}\ang{\phi}_{g^\inv(k)}}
  \cdot \si\ang{r_{\si(1)},\ldots,r_{\si(s)}}
  \cdot \si^\ell_{hg,f}.
  \]
  Next we use the bottom equivariance diagram \eqref{enr-operadic-eq-2} and have
  \[
  \ga\lrscmap{\mu_\ell;\ang{\om_k \cdot \si^k_{g,f}}_{k \in h^\inv(\ell)}} \cdot \si^\ell_{h,gf}
  =
  \ga\lrscmap{\mu_\ell; \ang{\om}_{h^\inv(\ell)}}
  \cdot \left(\;\bigoplus_{k \in h^\inv(\ell)} \si^k_{g,f}\right)
  \cdot \si^\ell_{h,gf}.
  \]

  To complete the verification that the two composites in \cref{eq:free-perm-assoc} are equal, observe the following for each $\ell \in \ufs{u}$.
  First, by associativity \eqref{multicategory-associativity} in $\M$ we have
  \begin{align*}
  \ga\lrscmap{\theta_\ell ; \bigoplus_{k \in h^\inv(\ell)} \ang{\phi}_{g^\inv(k)}}
  & = \ga\lrscmap{\ga\lrscmap{\mu_\ell; \ang{\psi}_{h^\inv(\ell)}} ; \bigoplus_{k \in h^\inv(\ell)} \ang{\phi}_{g^\inv(k)}}\\
  & = \ga\lrscmap{\mu_\ell;
     \left\langle\ga\lrscmap{\psi_k; \ang{\phi}_{g^\inv(k)}}\right\rangle_{k \in h^\inv(\ell)}}\\
  & = \ga\lrscmap{\mu_\ell;
    \ang{\om}_{h^\inv(\ell)}}
  \end{align*}
  in
  \[
  \M\lrscmap{
    \;\bigoplus_{k \in h^\inv(\ell)} \left(
    \;\bigoplus_{j \in g^\inv(k)} \ang{x}_{f^\inv(j)}
    \right)
    ;
    w_\ell
  }.
  \]
  Second, by uniqueness of the permutation in $\Si_{|(hgf^\inv)(\ell)|}$ whose right action sends 
  \[
  \bigoplus_{k \in h^\inv(\ell)} \left(
  \;\bigoplus_{j \in g^\inv(k)} \ang{x}_{f^\inv(j)}
  \right),
  \]
  to $\ang{x}_{(hgf)^\inv(\ell)}$, we have
  \[
  \si\ang{r_{\si(1)},\ldots,r_{\si(s)}}
  \cdot \si^\ell_{hg,f}
  =
  \left(
  \;\bigoplus_{k \in h^\inv(\ell)} \si^k_{g,f}
  \right) \cdot \si^\ell_{h,gf}.
  \]
  Therefore the two sides of \eqref{eq:free-perm-assoc} are equal for each triple of composable morphisms.
  This completes the proof that $F\M$ is a category.  The symmetric monoidal axioms for $F\M$ are verified as part of \cref{definition:free-perm}.
\end{proof}

\begin{definition}\label{definition:free-smfun}
  Suppose $H\cn \M \to \N$ is a multifunctor.
  Define a strict symmetric monoidal functor
  \[
  FH \cn F\M \to F\N
  \]
  via the following assignment on objects and morphisms.  For a sequence $\ang{x}$ of length $r$, define
  \[
  (FH)\ang{x} = \ang{Hx_i}_{i \in \ufs{r}}.
  \]
  For a morphism $(f,\ang{\phi})$, define
  \begin{equation}\label{eq:FH-fphi}
  (FH)(f,\ang{\phi}) = (f,\ang{H\phi_j}_{j}).
  \end{equation}

  The multifunctoriality of $H$ shows that this assignment is functorial on morphisms.
  Since the monoidal sum is defined by concatenation in $F\M$ and $F\N$, the functor $FH$ is strict monoidal.
  Compatibility with the symmetry of $F\M$ and $F\N$ follows because $FH$ preserves the index map of each morphism and $H$ preserves unit operations.
\end{definition}

\begin{definition}\label{definition:free-perm-multinat}
  Suppose $\kappa\cn H \to K \cn \M \to \N$ is a multinatural transformation.
  Define a monoidal natural transformation
  \[
  F\kappa \cn FH \to FK
  \]
  via components
  \begin{equation}\label{eq:Fka-x}
  (F\kappa)_{\ang{x}} = (1, \ang{\ka_{x_i}}_i) \cn \ang{Hx} \to \ang{Kx} 
  \end{equation}
  for each sequence $\ang{x}$ in $F\M$.  Naturality of $F\kappa$ follows from multinaturality of $\kappa$ \eqref{multinat} because each $\si^j_{f,1}$ and $\si^j_{1,f}$ is an identity permutation and we have
  \begin{align*}
  (1,\ang{\ka_{y_j}}_j)(f,\ang{H\phi_j}_j)
    & = \bigg(f, \big\langle
      \ga\scmap{\ka_{y_j} ; H\phi_j}
    \big\rangle_{j} \bigg)\\
    & = \bigg(f, \big\langle
      \ga\scmap{K\phi_j ; \ang{\ka_{x_i}}_{i \in f^\inv(j)}}
    \big\rangle_j\bigg)\\
    & = (f,\ang{K\phi_j}_j)(1,\ang{\ka_{x_i}}_i)
  \end{align*}
  for each morphism $(f,\ang{\phi})\cn \ang{x} \to \ang{y}$ in $F\M$.

  The monoidal naturality axioms \eqref{eq:monoidal-nt} for $F\ka$ follow because the monoidal sum in $F\N$ is given by concatenation of object and operation sequences.  The component $(F\ka)_{\ang{}}$ is the identity morphism $(1_{\varnothing},\ang{})\cn \ang{} \to \ang{}$.
\end{definition}

\begin{proposition}\label{proposition:free-perm-functor}
  The free permutative category construction given in \cref{definition:free-perm,definition:free-smfun} provides a 2-functor
  \[
  F \cn \Multicat \to \PermCat.
  \]
\end{proposition}
\begin{proof}
  Verification that $F$ produces permutative categories and strict symmetric monoidal functors is given in the definitions and \cref{proposition:free-perm}.
  We see that $F$ is functorial because, for each multifunctor $H$, the symmetric monoidal functor $FH$ is given by applying $H$ componentwise to sequences of objects and operations.  Therefore $F$ preserves identity morphisms and composition.  The 2-functoriality of $F$ follows similarly because the identities, horizontal, and vertical composites of both  multinatural transformations and monoidal natural transformations are determined by their components.
\end{proof}

\begin{example}\label{example:free-empty}
  The free permutative category on the empty multicategory is the terminal permutative category.  Its single object is the empty sequence, and its single morphism is given by the identity index map and the empty sequence of operations.
\end{example}

\begin{example}\label{example:free-terminal}
  Recall from \cref{def:multicategory} that the terminal multicategory $\Mterm$ consists of a single object and a unique $n$-ary operation for each $n$.
  The free permutative category $F\Mterm$ is isomorphic to the natural number category $\N$ whose objects are given by natural numbers and morphisms are given by morphisms of finite sets
  \[
  \N(r,s) = \Set(\ufs{r},\ufs{s}).
  \]
  The natural number $r \in \N$ corresponds to the length-$r$ sequence whose terms are the unique object of $\Mterm$.  Each morphism $f\cn \ufs{r} \to \ufs{s}$ corresponds to the morphism
  \[
  (f,\ang{\phi}) \in F\Mterm
  \]
  where $\phi_j$ is the unique operation in $\Mterm$ of arity $|f^\inv(j)|$.
\end{example}

\begin{example}
  Recall from \cref{def:multicategory} that the initial operad $\Mtu$ consists of a single object and its unit operation.  Similarly to \cref{example:free-terminal}, the free permutative category $F\Mtu$ is isomorphic to the permutation category $\Si$ with objects given by natural numbers and morphisms given by permutations
  \[
  \Si(r,s) = \begin{cases}
    \Si_r, & \ifspace r = s,\\
    \varnothing, & \ifspace r \not= s
  \end{cases}
  \]
  for each pair of natural numbers $r$ and $s$.
\end{example}

\section{Free Permutative Category as a Left 2-Adjoint}\label{sec:free-perm-adj}
Recall from \cref{proposition:end-2fun} that the endomorphism multicategory $\End(\C)$ associated to a permutative category $\C$ defines a 2-functor
\[
\End\cn \PermCat \to \Multicat.
\]
Throughout this section we let $E = \End$.  The main result of this section is \cref{theorem:FE-adj}, which shows that the free construction $F$ of \cref{sec:free-perm} is left 2-adjoint to $E = \End$.

\begin{definition}[Unit]\label{definition:eta}
  Suppose $\M$ is a multicategory.
  Define a component
  \[
  \eta = \eta_{\M}\cn \M \to EF\M
  \]
  as follows.
  For an object $w \in \M$ and an operation $\phi \in \M\mmap{y;\ang{x}}$, let $(w)$ and $(\phi)$ denote the corresponding length-1 sequences.
  For each $r \ge 0$, let $\iota_r \cn \ufs{r} \to \ufs{1}$ be the unique function.
  Then $\eta = \eta_\M$ is the following assignment:
  \begin{align*}
    \eta w & = (w) \forspace w \in \M, \andspace\\
    \eta \phi & = (\iota_r, (\phi)) \cn \ang{x} \to (y)
  \end{align*}
  where $\phi \in \M\mmap{y;\ang{x}}$ and $|\ang{x}| = r$.
  Note that $\ang{x}$ is an $r$-fold concatenation of length-1 sequences $(x_i)$ and the morphism $(\iota_r, (\phi))$ in $F\M$ is an $r$-ary operation in $EF\M$.
  \Cref{lemma:eta-nat} shows that each $\eta_\M$ is multinatural and that the components are 2-natural with respect to multifunctors and multinatural transformations.
\end{definition}

\begin{lemma}\label{lemma:eta-nat}
  The components $\eta_\M$ of \cref{definition:eta} define a 2-natural transformation
  \[
  \eta\cn 1_{\Multicat} \to EF.
  \]
\end{lemma}
\begin{proof}
  To check multifunctoriality of each component $\eta = \eta_{\M}$, first note that $\eta$ preserves unit operations because $\iota_{1}$ is the identity on $\ufs{1}$.  Now suppose given
  \[
  \psi \in \M\mmap{x'';\ang{x'}} \wherespace |\ang{x'}| = s, \andspace
  \]
  \[
  \phi_j \in \M\mmap{x'_j;\ang{x_j}} \wherespace |\ang{x_j}| = r_j \forspace
  j \in \ufs{s}.
  \]
  Let $r = \sum_j r_j$.
  Then the composite in $EF\M$ of
  \[
  \eta \psi = (\iota_s; (\psi)) \andspace \ang{\eta \phi_j}_{j=1}^s = \ang{(\iota_{r_j}, (\phi_j))}_j
  \]
  is given by composing the morphisms
  \begin{equation}\label{eq:eta-multinat-comp}
  (\iota_s, (\psi)) \andspace \bigoplus_{j \in \ufs{s}} (\iota_{r_j}, \phi_j) = \big(\oplus_{j \in \ufs{s}} \iota_{r_j}, \ang{\phi}\big)
  \end{equation}
  in $F\M$.
  For $g = \iota_s$ and $f = \oplus_j \iota_{r_j}$, the permutation $\sigma^1_{g,f}$ of \eqref{eq:sigma-kgf} is the identity on $\ufs{r}$.
  Therefore the composite of the morphisms \eqref{eq:eta-multinat-comp} is
  \[
  \big(\iota_s \circ (\oplus_j \iota_{r_j}), \ga\scmap{\psi;\ang{\phi}}\big)
  = \big(\iota_r, \ga\scmap{\psi;\ang{\phi}}\big) 
  = \eta\ga\scmap{\psi; \ang{\phi}}.
  \]

  Naturality of $\eta$ with respect to multifunctors $H$ follows because $FH$ is given by applying $H$ termwise to sequences of objects and operations.
  The index map of each morphism $(f,\ang{\phi})$ is preserved by $FH$ as in \cref{eq:FH-fphi}.
  Similarly, 2-naturality of $\eta$ with respect to multinatural transformations $\ka$ follows because the components of $F\ka$ have the identity index map and take sequences of the corresponding components of $\ka$ as in \cref{eq:Fka-x}.
  This completes the proof that
  \[
  \eta\cn 1_{\Multicat} \to EF
  \]
  is a 2-natural transformation.
\end{proof}

\begin{definition}[Counit]\label{definition:epz}
  Suppose $\C$ is a permutative category.
  Define a component
  \[
  \epz = \epz_\C \cn FE\C \to \C
  \]
  as follows.
  For each index map
  \[
  f\cn \ufs{r} \to \ufs{s}
  \]
  and each length-$r$ sequence of objects $\ang{x}$ with $x_i \in \C$, let
  \begin{equation}\label{eq:xi-f}
    \bigoplus_{i \in \ufs{r}} x_i \fto{\xi_f} \bigoplus_{j \in \ufs{s}} \bigoplus_{i \in f^\inv(j)} x_i
  \end{equation}
  denote the unique morphism in $\C$ given by components of $\xi$ permuting the terms of the sum.
  Uniqueness of this morphism in $\C$ follows from the Symmetric Coherence Theorem \cite[XI.1, Th.\ 1]{maclane}.
  Then $\epz$ is the following assignment:
  \begin{align}
    \epz\ang{x} & = \bigoplus_i x_i \forspace \ang{x} \in FE\C \andspace\label{eq:xi-x}\\
    \epz(f,\ang{\phi}) & = \Big(\bigoplus_j \phi_j \Big) \circ \xi_f\label{eq:xi-fphi}
  \end{align}
  where $(f,\ang{\phi})\cn \ang{x} \to \ang{y}$ in $FE\C$.
  \Cref{lemma:epz-nat} shows that each $\epz_\C$ is a strict symmetric monoidal functor and that the components are 2-natural with respect to strict symmetric monoidal functors and monoidal natural transformations.
\end{definition}

\begin{lemma}\label{lemma:epz-nat}
  The components $\epz_\C$ of \cref{definition:epz} define a 2-natural transformation
  \[
  \epz\cn FE \to 1_{\permcat}.
  \]
\end{lemma}
\begin{proof}
  To check that each component $\epz = \epz_{\C}$ is a strict symmetric monoidal functor, first note that $\epz(1,\ang{1_{x_i}}_i)$ is the identity on $\oplus_i x_i$.
  Now suppose given a composable pair of morphisms in $FE\C$
  \[
  \ang{x} \fto{(f,\ang{\phi})} \ang{y} \fto{(g,\ang{\psi})} \ang{z}
  \]
  where $|\ang{x}| = r$, $|\ang{y}| = s$, and $|\ang{z}| = t$.
  The composite $(\epz g)(\epz f)$ appears along the sides and top of the following diagram, while $\epz(gf)$ appears along the bottom with $\theta_k = \ga\scmap{\psi_k; \ang{\phi}_{g^\inv(k)}}$ as in \eqref{eq:thetak},
  \[
  \begin{tikzpicture}[x=22mm,y=30mm]
    \draw[0cell] 
    (0,0) node (a) {\bigoplus_{i \in \ufs{r}} x_i}
    (a) ++(0,1) node (b) {\bigoplus_{j \in \ufs{s}} \; \bigoplus_{i \in f^\inv(j)} x_i}
    (b) ++(2,0) node (c) {\bigoplus_{j \in \ufs{s}} y_j}
    (c) ++(2,0) node (d) {\bigoplus_{k \in \ufs{t}} \; \bigoplus_{j \in g^\inv(j)} y_j}
    (d) ++(0,-1) node (e) {\bigoplus_{k \in \ufs{t}} z_k}
    (a) ++(2,0) node (f) {\bigoplus_{k \in \ufs{t}} \; \bigoplus_{i \in (gf)^\inv(k)} x_i}
    (c) ++(0,-.5) node (g) {
      \bigoplus_{k \in \ufs{t}}  \; \bigoplus_{j \in g^\inv(k)} \; \bigoplus_{i \in f^\inv(j)} x_i
    }
    ;
    \draw[1cell] 
    (a) edge node {\xi_f} (b)
    (b) edge node {\oplus_j \phi_j} (c)
    (c) edge node {\xi_g} (d)
    (d) edge node {\oplus_k \psi_k} (e)
    (a) edge['] node {\xi_{gf}} (f)
    (f) edge['] node {\oplus_k (\theta_k \cdot \si^k_{gf})} (e)
    (a) edge node {} (g)
    (b) edge node {} (g)
    (f) edge node {} (g)
    (g) edge['] node {\oplus_k \oplus_{j} \phi_j} (d)
    (g) edge node {\oplus_k \theta_k} (e)
    ;
  \end{tikzpicture}
  \]
  In the above diagram, each of the three unlabeled morphisms is the unique morphism in $\C$ given by components of $\xi$ permuting terms.
  The two triangles at left and bottom-left commute by the Symmetric Coherence Theorem.
  The triangle at bottom-right commutes by definition of $\si^k_{g,f}$ in \eqref{eq:sigma-kgf}.
  The triangle at top commutes by naturality of $\xi$ and the triangle at right commutes by definition of $\theta_k$.
  Therefore each component $\epz_\C$ is functorial.

  The monoidal sum in $FE\C$ is given by concatenation of sequences and disjoint union of index maps.
  So $\epz_\C$ is a strict monoidal functor because
  \begin{itemize}
  \item the monoidal sum in $\C$ is strictly associative and functorial,
  \item the morphisms $\xi_f$ are uniquely determined by the index maps $f$, and
  \item the empty monoidal sum in $\C$ is the monoidal unit.
  \end{itemize}
  Therefore we have
  \[
  \epz\ang{x} \oplus \epz\ang{x'} = \bigg(\bigoplus_i x_i\bigg) \oplus \bigg(\bigoplus_{i'} x'_i\bigg)
  = \epz(\ang{x} \oplus \ang{x'})
  \]
  and
  \[
  \epz(f,\ang{\phi}) \oplus \epz(f',\ang{\phi'})
  = \bigg(\big(\oplus_j \phi\big) \circ \xi_f\bigg) \oplus \bigg(\big(\oplus_{j'} \phi'\big) \circ \xi_{f'}\bigg)
  = \epz(f \oplus f' , \phi \oplus \phi')
  \]
  for objects $\ang{x}$, $\ang{x'}$ and morphisms $(f,\ang{\phi})$, $(f',\ang{\phi'})$ in $FE\C$.
  Moreover, $\epz_\C$ is a symmetric monoidal functor because the symmetry in $FE\C$ is given by block permutation of sequences and hence $\epz$ sends the symmetry of $FE\C$ to that of $\C$.
  
  Naturality of $\epz$ with respect to strict symmetric monoidal functors
  \[
  P\cn \C \to \D
  \]
  is commutativity of the following square for each such $P$.
  \begin{equation}\label{eq:epz-nat-square}
  \begin{tikzpicture}[x=20mm,y=15mm,vcenter]
    \draw[0cell] 
    (0,0) node (a) {FE\C}
    (1,0) node (b) {\C}
    (0,-1) node (c) {FE\D}
    (1,-1) node (d) {\D}
    ;
    \draw[1cell] 
    (a) edge node {\epz_\C} (b)
    (c) edge node {\epz_\D} (d)
    (a) edge['] node {FEP} (c)
    (b) edge node {P} (d)
    ;
  \end{tikzpicture}
  \end{equation}
  The above square commutes on objects $\ang{x}$ in $FE\C$ because $P$ is strict monoidal and hence
  \[
  P\left(\oplus_i x_i\right) = \oplus_i Px_i.
  \]
  Commutativity on morphisms $(f,\ang{\phi})$ depends on the following.
  \begin{itemize}
  \item By \cref{definition:free-smfun}, $FEP$ does not change the index map $f$.
  \item Since $P$ is a strict symmetric monoidal functor, we have
    \[
    P(\xi^\C_f) = \xi^\D_f
    \]
    where $\xi^\C_f$ and $\xi^\D_f$ are the unique morphisms of \eqref{eq:xi-f} induced by the symmetry of $\C$ and $\D$, respectively.
  \item Since $P$ is functorial, it preserves the composition of morphisms in the definition of $\epz$ \eqref{eq:xi-fphi}.
  \end{itemize}

  For 2-naturality of $\epz$, suppose
  \[
  \al \cn \P \to \Q \cn \C \to \D
  \]
  is a monoidal natural transformation between strict symmetric monoidal functors.
  The compatibility with monoidal constraints in \eqref{eq:monoidal-nt} implies that the components of $\al$ preserve monoidal sums.
  Therefore, using \cref{eq:Fka-x}, we have
  \[
  \epz(1,\ang{\al_{x_i}}) = \bigoplus_i \al_{x_i} = \al_{\oplus_i x_i} = \al_{\epz\ang{x}}
  \]
  for each $\ang{x} \in FE\C$.
  This completes the proof that
  \[
  \epz\cn FE \to 1_{\PermCat}
  \]
  is a 2-natural transformation.
\end{proof}
\begin{remark}\label{remark:epz-2nat-strictness}
  Note that the 2-naturality of $\epz$ depends on the assumption that $P$ is a \emph{strict} monoidal functor.
  Without that assumption, the naturality square \eqref{eq:epz-nat-square} generally does not commute on objects.
  This requirement is one of our motivations for restricting the morphisms of $\PermCat$ to be strict symmetric monoidal functors.
\end{remark}

\begin{theorem}\label{theorem:FE-adj}
  There is a 2-adjunction
  \[
  F \cn \Multicat \lradj \PermCat \cn \End
  \]
  with $F \dashv \End$.
\end{theorem}
\begin{proof}
  The unit and counit
  \[
  \eta\cn 1_{\Multicat} \to EF
  \andspace
  \epz\cn FE \to 1_{\PermCat}
  \]
  are shown to be 2-natural transformations in \cref{lemma:eta-nat,lemma:epz-nat}.
  For each small multicategory $\M$, the composite
  \begin{equation}\label{eq:triang-FEF}
  F\M \fto{F \eta_\M} FEF\M \fto{\epz_{F\M}} F\M
  \end{equation}
  is the identity on objects because each sequence $\ang{x}$ is equal to the concatenation of length-1 sequences of its entries.
  For a morphism $(f,\ang{\phi})$ in $F\M$, 
  \[
  (\epz_{F\M}) (F\eta_\M) (f, \ang{\phi})
  = \epz_{F\M} \left(f, \ang{ ( \iota_{|f^{\inv}(j)|}, \phi_j ) }_{j \in \ufs{s}} \right)
  \]
  is the composite of the two morphisms in the following diagram.
  \[
    \begin{tikzpicture}[x=55mm,y=20mm]
      \draw[0cell] 
      (0,0) node (a) {\ang{x}}
      (1,0) node (b) {\ang{y}}
      (.5,-.5) node (c) {\ang{\ang{x_i}_{i \in f^\inv(j)}}_{j \in \ufs{s}}}
      ;
      \draw[1cell] 
      (a) edge node {
        \epz_{F\M} \left(
          f, \ang{ ( \iota_{|f^{\inv}(j)|}, \phi_j ) }_{j \in \ufs{s}}
        \right)
      } (b)
      (a) edge['] node {(\xi_f, \ang{1})} (c)
      (c) edge['] node {\big(
        \bigoplus_{j \in \ufs{s}} \iota_{|f^{\inv}(j)|}, \ang{\phi_j}_{j \in \ufs{s}}
        \big)} (b)
      ;
    \end{tikzpicture}
  \]
  The index map for this composite is
  \[
    \ufs{r} \fto{\xi_f} \ufs{r} \fto{\oplus_j\; \iota_{|f^\inv(j)|}} \ufs{s},
  \] 
  which is equal to $f$.  The second component of the composite is $\ang{\phi}$ because, for each $k \in \ufs{s}$, the permutation $\sigma^k_{\oplus_j\, \iota_{|f^\inv(j)|} \,,\, \xi_f}$ of \eqref{eq:sigma-kgf} is the identity.
  
  A similar checking of objects and morphisms shows that the composite
  \[
  E\C \fto{\eta_{E\C}} EFE\C \fto{E \epz_\C} E\C
  \]
  is the identity for each small permutative category $\C$
  because each $\xi_{\iota_r}$ is an identity morphism.  This completes the proof that $F \dashv E$ is a 2-adjunction.
\end{proof}

We close this section with one further construction: a componentwise right-adjoint for $\epz$.  This will be used in the proof of \cref{theorem:F-End-hthy-equiv} to show that $\epz$ is a componentwise stable equivalence.
\begin{proposition}\label{proposition:epz-rho-adj}
  For each permutative category $\C$ there is an adjunction of categories
  \[
  \epz_\C\cn FE\C \lradj \C\cn \rho_\C
  \]
  where $\rho_\C$ is induced by inclusion of length-1 tuples and $\epz_\C \dashv \rho_C$.
\end{proposition}
\begin{proof}
  Throughout this proof we write $\epz = \epz_\C$ and $\rho = \rho_\C$.
  The functor $\rho$ is defined by the assignments
  \begin{align*}
    \rho x & = (x) \forspace x \in \C \andspace\\
    \rho \phi & = (1_{\ufs{1}},(\phi))
  \end{align*}
  for morphisms $\phi$ in $\C$.
  In this definition both $(x)$ and $(\phi)$ denote length-1 tuples.
  The composite
  \[
  \C \fto{\rho} FE\C \fto{\epz} \C
  \]
  is the identity.
  The composite
  \[
  FE\C \fto{\epz} \C \fto{\rho} FE\C
  \]
  is given by the assignments
  \begin{align*}
    \ang{x} & \mapsto \big(\oplus_i x_i\,\big)\\
    (f,\ang{\phi}) & \mapsto \big(1_{\ufs{1}}\,,\,\big(\oplus_j \phi_j\big)\circ \xi_f\,\big)
  \end{align*}
  where $\ang{x}$ and $(f,\ang{\phi})$ are objects and morphisms in $FE\C$.
  A unit for the adjunction $(\epz,\rho)$ is provided by the following.
  Define a natural transformation $\al\cn 1_{FE\C} \to \rho \epz$ with components
  \begin{equation}\label{eq:al-x}
  \al_{\ang{x}} = \big( \iota_r, 1_{\oplus_i x_i} \big) \cn \ang{x} \to \big(\oplus_i x_i\, \big)
  \end{equation}
  for a length-$r$ sequence in $\ang{x}$ in $FE\C$.
  In \eqref{eq:al-x}, $\iota_r\cn \ufs{r} \to \ufs{1}$ denotes the unique morphism.
  For a morphism
  \[
    (f,\ang{\phi})\cn \ang{x} \to \ang{y} \in FE\C,
  \]
  where $\ang{x}$ has length $r$ and $\ang{y}$ has length $s$, the naturality diagram for $\al$ is the following.
  \begin{equation}\label{eq:nat-al}
    \begin{tikzpicture}[x=50mm,y=15mm,vcenter]
      \draw[0cell] 
      (0,0) node (a) {\ang{x}}
      (1,0) node (b) {\ang{y}}
      (0,-1) node (c) {\big(\oplus_{i} x_i\big)}
      (1,-1) node (d) {\big(\oplus_{j} y_j\big)}
      ;
      \draw[1cell] 
      (a) edge node {(f, \ang{\phi})} (b)
      (c) edge node {\big(1_{\ufs{1}} \,,\, \big(\oplus_j \phi_j\big) \circ \xi_f\big)} (d)
      (a) edge['] node {\big( \iota_r, 1_{\oplus_i x_i} \big)} (c)
      (b) edge node {\big( \iota_s, 1_{\oplus_j y_j} \big)} (d)
      ;
    \end{tikzpicture}
  \end{equation}
  In the above diagram, the first component in each composite is $\iota_r$.  The second component in each composite is $(\oplus_j \phi_j) \circ \xi_f$ because the right action of $\si^1_{\iota_s,f}$ on $\oplus_j \phi_j$ is equal to the composition with $\xi_f$.

  The triangle identities for unit $\al$ and counit $1_{1_\C}$ hold as follows.
  First, for an object $x$ of $\C$, the component $\al_{(x)}$ is the identity on $(x)$.
  Second, for an object $\ang{x}$ of $FE\C$, the morphism $\epz(\al_{\ang{x}})$ is the identity on $\oplus_i x_i$.
  This finishes the proof that $(\epz,\rho) = (\epz_\C,\rho_\C)$ is an adjunction of categories.
\end{proof}
\begin{remark}
  Note that the components $\rho_\C$ in the proof of \cref{proposition:epz-rho-adj} are not strictly monoidal, because the monoidal sum in $FE\C$ is given by concatenation.
  For objects $x$ and $x'$ in $\C$, there is a monoidal constraint morphism
  \[
  (\iota_{2}, 1_{x \oplus x'})\cn (x,x') \to (x \oplus x').
  \]
  There is also a unit constraint morphism
  \[
  (\iota_0,1_e)\cn \ang{} \to (e).
  \]
  One can verify that these satisfy associativity, unit, and symmetry axioms to make $\rho_\C$ a symmetric monoidal functor.
\end{remark}

\section{Multicategories Model All Connective Spectra}\label{sec:multicats-model} 

For permutative categories and for multicategories, we define stable equivalences via the stable equivalences on $K$-theory spectra.  We let $\SymSp$ denote the Hovey-Shipley-Smith category of symmetric spectra \cite{hss}.  We let
\[
K \cn \PermCat \to \SymSp
\]
denote Segal's $K$-theory functor \cite{segal} that constructs a connective symmetric spectrum from each small permutative category.
See \cite[Chapters~7 and~8]{cerberusIII} for a review and further references.
For the work below we will need the following two facts about stable equivalences.
\begin{remark}[Stable Equivalences]\label{remark:steq}
  \ 
\begin{enumerate}
\item There is a Quillen model structure on $\SymSp$ whose weak equivalences are the stable equivalences \cite[3.4.4 and~5.3.8]{hss}.
  In particular, the class of stable equivalences includes isomorphisms, is closed under composition, and has the 2-out-of-3 property.
\item\label{it:steq-2} If $P\cn \C \to \D$ is a strict symmetric monoidal functor whose underlying functor is a left or right adjoint, then $KP$ is a stable equivalence \cite[7.2.5 and~7.8.8]{cerberusIII}.\dqed
\end{enumerate}
\end{remark}

\begin{definition}\label{definition:she}
  A \emph{stable equivalence} between permutative categories is a strict symmetric monoidal functor $P$ such that $KP$ is a stable equivalence of $K$-theory spectra.
  A \emph{stable equivalence} between multicategories is a multifunctor $H$ such that $FH$ is a stable equivalence of permutative categories.
  Equivalently, the stable equivalences of permutative categories are reflected by $K$ and the stable equivalences of multicategories are reflected by $F$.
\end{definition}

\begin{theorem}\label{theorem:F-End-hthy-equiv}
  The free functor $F$ and the endomorphism functor $\End$ induce
  equivalences of homotopy theories
  \[
  F\cn \big(\Multicat,\cS\big) \lrsimadj \big(\PermCat,\cS\big) \cn \End
  \]
  where $\cS$ denotes the class of stable equivalences in each category.
\end{theorem}
\begin{proof}
  By definition of stable equivalences, $F$ is a relative functor.
  To see that $\End$ is a relative functor, we first note that the components of $\epz$ are stable equivalences by
  \Cref{proposition:epz-rho-adj} and \cref{remark:steq}~(\cref{it:steq-2}).
  Naturality of $\epz$ and the 2-out-of-3 property for stable equivalences then imply that
  \[
  F\End(P)\cn F\End(\C) \to F\End(\D)
  \]
  is a stable equivalence whenever $P$ is a stable equivalence.
  This, in turn, implies that $\End(P)$ is a stable equivalence.
  Hence $\End$ is a relative functor.

  The triangle identities for $\eta$ and $\epz$, together with the 2-out-of-3 property, imply that the components of $\eta$ are also stable equivalences.  Then the result follows from \cref{gjo29}.
\end{proof}

\bibliographystyle{sty/amsalpha3}
\bibliography{references}

\providecommand{\bysame}{\leavevmode\hbox to3em{\hrulefill}\thinspace}
\providecommand{\MR}{\relax\ifhmode\unskip\space\fi MR }
\providecommand{\MRhref}[2]{%
  \href{http://www.ams.org/mathscinet-getitem?mr=#1}{#2}
}
\providecommand{\nopubyear}{$\infty$}
\providecommand{\doi}[1]{%
  doi:\href{https://dx.doi.org/#1}{\nolinkurl{#1}}}
\providecommand{\arxiv}[1]{%
  arXiv:\href{https://arxiv.org/abs/#1}{#1}}
\begin{thebibliography}{Yau$\infty$IIAA}

\bibitem[BK12]{barwick-kan}
C.~Barwick and D.~M. Kan, \emph{A characterization of simplicial localization
  functors and a discussion of {DK} equivalences}, Indag. Math. (N.S.)
  \textbf{23} (2012), no.~1-2, 69--79. \doi{10.1016/j.indag.2011.10.001}

\bibitem[BV73]{boardman-vogt}
J.~M. Boardman and R.~M. Vogt, \emph{Homotopy invariant algebraic structures on
  topological spaces}, Lecture Notes in Mathematics, Vol. 347, Springer-Verlag,
  Berlin-New York, 1973.

\bibitem[BO20]{bohmann_osorno}
A.~M. Bohmann and A.~M. Osorno, \emph{A multiplicative comparison of {S}egal
  and {W}aldhausen {$K$}-theory}, Math. Z. \textbf{295} (2020), no.~3-4,
  1205--1243. \doi{10.1007/s00209-019-02394-7}

\bibitem[DK80]{dwyer-kan}
W.~G. Dwyer and D.~M. Kan, \emph{Calculating simplicial localizations}, J. Pure
  Appl. Algebra \textbf{18} (1980), no.~1, 17--35.
  \doi{10.1016/0022-4049(80)90113-9}

\bibitem[EM06]{elmendorf-mandell}
A.~D. Elmendorf and M.~A. Mandell, \emph{Rings, modules, and algebras in
  infinite loop space theory}, Adv. Math. \textbf{205} (2006), no.~1, 163--228.
  \doi{10.1016/j.aim.2005.07.007}

\bibitem[EM09]{elmendorf-mandell-perm}
\bysame, \emph{Permutative categories, multicategories and algebraic
  {$K$}-theory}, Algebr. Geom. Topol. \textbf{9} (2009), no.~4, 2391--2441.
  \doi{10.2140/agt.2009.9.2391}

\bibitem[GJO17a]{gjo-extending}
N.~Gurski, N.~Johnson, and A.~M. Osorno, \emph{Extending homotopy theories
  across adjunctions}, Homology Homotopy Appl. \textbf{19} (2017), no.~2,
  89--110. \doi{10.4310/HHA.2017.v19.n2.a6}

\bibitem[GJO17b]{gjo1}
\bysame, \emph{{$K$}-theory for 2-categories}, Adv. Math. \textbf{322} (2017),
  378--472. \doi{10.1016/j.aim.2017.10.011}

\bibitem[Hir03]{hirschhorn}
P.~S. Hirschhorn, \emph{Model categories and their localizations}, Mathematical
  Surveys and Monographs, vol.~99, American Mathematical Society, Providence,
  RI, 2003.

\bibitem[HSS00]{hss}
M.~Hovey, B.~Shipley, and J.~Smith, \emph{Symmetric spectra}, J. Amer. Math.
  Soc. \textbf{13} (2000), no.~1, 149--208. \doi{10.1090/S0894-0347-99-00320-3}

\bibitem[JY$\infty$]{cerberusIII}
N.~Johnson and D.~Yau, \emph{{B}imonoidal {C}ategories, {$E_n$}-{M}onoidal
  {C}ategories, and {A}lgebraic {$K$}-{T}heory. {V}olume {III}: {F}rom
  {C}ategories to {S}tructured {R}ing {S}pectra}, available at
  \url{https://nilesjohnson.net/}.

\bibitem[JY21]{johnson-yau}
\bysame, \emph{{2}-{D}imensional {C}ategories}, Oxford University Press, New
  York, 2021. \doi{10.1093/oso/9780198871378.001.0001}

\bibitem[JS93]{joyal-street}
A.~Joyal and R.~Street, \emph{Braided tensor categories}, Adv. Math.
  \textbf{102} (1993), no.~1, 20--78. \doi{10.1006/aima.1993.1055}

\bibitem[ML98]{maclane}
S.~Mac~Lane, \emph{Categories for the working mathematician}, second ed.,
  Graduate Texts in Mathematics, vol.~5, Springer-Verlag, New York, 1998.

\bibitem[{Man}10]{mandell_inverseK}
M.~A. {Mandell}, \emph{{An inverse {$K$}-theory functor}}, {Doc. Math.}
  \textbf{15} (2010), 765--791.

\bibitem[Rez01]{rezk-homotopy-theory}
C.~Rezk, \emph{A model for the homotopy theory of homotopy theory}, Trans.
  Amer. Math. Soc. \textbf{353} (2001), no.~3, 973--1007 (electronic).
  \doi{10.1090/S0002-9947-00-02653-2}

\bibitem[Seg74]{segal}
G.~Segal, \emph{Categories and cohomology theories}, Topology \textbf{13}
  (1974), 293--312. \doi{10.1016/0040-9383(74)90022-6}

\bibitem[Tho95]{thomason}
R.~W. Thomason, \emph{Symmetric monoidal categories model all connective
  spectra}, Theory Appl. Categ. \textbf{1} (1995), No. 5, 78--118.

\bibitem[To{\"e}05]{toen-axiomatisation}
B.~To{\"e}n, \emph{Vers une axiomatisation de la th\'eorie des cat\'egories
  sup\'erieures}, $K$-Theory \textbf{34} (2005), no.~3, 233--263.
  \doi{10.1007/s10977-005-4556-6}

\bibitem[Yau$\infty$I]{cerberusI}
D.~Yau, \emph{{B}imonoidal {C}ategories, {$E_n$}-{M}onoidal {C}ategories, and
  {A}lgebraic {$K$}-{T}heory. {V}olume {I}: {S}ymmetric {B}imonoidal
  {C}ategories and {M}onoidal {B}icategories}, available at
  \url{https://u.osu.edu/yau.22/main/}.

\bibitem[Yau$\infty$II]{cerberusII}
\bysame, \emph{{B}imonoidal {C}ategories, {$E_n$}-{M}onoidal {C}ategories, and
  {A}lgebraic {$K$}-{T}heory. {V}olume {II}: {B}raided {B}imonoidal
  {C}ategories with {A}pplications}, available at
  \url{https://u.osu.edu/yau.22/main/}.

\bibitem[Yau16]{yau-operad}
\bysame, \emph{Colored operads}, Graduate Studies in Mathematics, vol. 170,
  American Mathematical Society, Providence, RI, 2016.

\bibitem[Zak18]{zakharevich}
I.~Zakharevich, \emph{The category of {W}aldhausen categories is a closed
  multicategory}, New directions in homotopy theory, Contemp. Math., vol. 707,
  Amer. Math. Soc., Providence, RI, 2018, pp.~175--194.
  \doi{10.1090/conm/707/14259}

\end{thebibliography}
\end{document}